\documentclass[12pt]{article}

\usepackage{amsmath}
\usepackage{amssymb}
\usepackage{amsthm}
\usepackage{url}
\usepackage{graphicx}

\newtheorem{theorem}{Theorem}

\newtheorem{lemma}[theorem]{Lemma}

\title{A characterisation of alternating knot exteriors}
\author{Joshua Howie} 
\date{}

\begin{document}

\maketitle

\begin{abstract}
We give a topological characterisation of alternating knot exteriors based on the presence of special spanning surfaces. This shows that alternating is a topological property of the knot exterior and not just a property of diagrams, answering an old question of Fox. We also give a characterisation of alternating link exteriors which have marked meridians. We then describe a normal surface algorithm which can decide if a knot is prime and alternating given a triangulation of its exterior as input.
\end{abstract}

\renewcommand{\thesubsection}{\arabic{subsection}}
\setcounter{theorem}{0}

\subsection{Introduction}\label{introd}

Let $L$ be a link in $S^3$, and let $N(L)$ be a regular open neighbourhood. Then the link exterior $X\cong S^3\setminus N(L)$ is a compact $3$-manifold with torus boundary components. A planar link diagram $\pi(L)$ is a projection
\[\pi:S^2\times I\rightarrow S^2,\]
where $L$ has been isotoped to lie in some $S^2\times I\subset S^3$, together with crossing information.
A diagram $\pi(L)$ is alternating if the crossings alternate between over- and under-crossings as we traverse the projection of the link. A non-trivial link is alternating if it admits an alternating diagram.
We take the convention that the unknot is not alternating.

A simple Euler characteristic argument shows that if $\pi(L)$ is a planar diagram with $n$ crossings and checkerboard surfaces $\Sigma$ and $\Sigma'$, then
\[\chi(\Sigma)+\chi(\Sigma')+n=2.\]
Furthermore, if $\pi(L)$ is reduced, non-split and alternating, then $\Sigma$ and $\Sigma'$ are both $\pi_1$-essential in $X$ and $2n$ is the difference between the aggregate slopes of $\Sigma$ and $\Sigma'$. If $K$ is a knot, then $2n$ is the difference between their boundary slopes.

Our main result is to prove the converse, where we think of the difference in boundary slopes of $\Sigma$ and $\Sigma'$ as the minimal geometric intersection number of $\partial\Sigma$ and $\partial\Sigma'$ on $\partial X$, which we denote by $i(\partial\Sigma,\partial\Sigma')$.

\newtheorem*{chars3}{Theorem~\ref{chars3}}
\begin{chars3}
Let $K$ be a non-trivial knot in $S^3$ with exterior $X$. Then $K$ is alternating if and only if there exist a pair of connected spanning surfaces $\Sigma$, $\Sigma'$ in $X$ which satisfy
\[\chi(\Sigma)+\chi(\Sigma')+\frac{1}{2}i(\partial\Sigma,\partial\Sigma')=2.\tag{$\star$}\]
\end{chars3}

Notice that all the conditions on one side of this characterisation are topological in nature. Theorem~\ref{chars3} answers an old question, attributed to Ralph Fox, ``What is an alternating knot?'' This question has been interpreted as requesting a non-diagrammatic description of alternating knots~\cite{lickbk}.

By a theorem of Gordon and Luecke~\cite{gorlue}, a knot exterior has a unique meridian so the concept of a spanning surface is well-defined in $X$. For link exteriors, this is not true~\cite{gorlink} and there are $3$-manifolds which are homeomorphic to the exterior of both alternating and non-alternating links. However, $X$ together with a marked meridian on each boundary component does uniquely determine a link in $S^3$. With the addition of an extra condition on intersection numbers, we are able to give a characterisation of alternating link exteriors with marked meridians in Theorem~\ref{chars4}.

For the second half of this article we turn to normal surface theory, and show that given a $3$-manifold with connected torus boundary $X$, it is possible to decide if $X$ is the exterior of a prime and alternating knot in $S^3$. We also show that given a non-alternating planar diagram of a knot $K$, we can decide if $K$ is prime and alternating, and if so, we can produce a prime alternating diagram of either $K$ or its mirror image.

\newtheorem*{algdec}{Theorem~\ref{algdec}}
\begin{algdec}
Let $X$ be the exterior of a knot $K\subset S^3$. Given $X$, there is an algorithm to decide if $K$ is prime and alternating.
\end{algdec}

In Section~\ref{intsurf}, we describe how two spanning surfaces for a link intersect. In Section~\ref{charac}, we prove Theorem~\ref{chars3}, and give the version for links. In Section~\ref{normsurf}, we give some background on normal surface theory and the boundary solution space. In Section~\ref{altalg}, we detail the algorithm which can decide if a knot manifold is the exterior of an alternating knot.

\paragraph*{Acknowledgement.}

The author would like to thank Hyam Rubinstein for many interesting conversations and help with the algorithm.

\subsection{Intersections of Spanning Surfaces}\label{intsurf}

Let $L$ be a link with $m$ components denoted by $L_j$ for $j=1,\ldots,m$. Denote the boundary components of $X$ by $C_j=\partial N(L_j)$ where each $C_j$ is a torus.

A curve $\mu_j\subset C_j$ is a meridian of $X$ if $\mu_j$ bounds an embedded disk in $N(L_j)$ which intersects $L_j$ transversely exactly once. Given $X$, a set of marked meridians is a set of curves $\{\mu_j\}$ with one $\mu_j$ on each $C_j$ such that Dehn filling along each $\mu_j$ produces the $3$-sphere.

We define a preferred longitude $\lambda_j$ of $C_j$ to be the unique non-trivial curve on $C_j$ which meets $\mu_j$ exactly once and bounds an orientable surface $S_j$ in $S^3\setminus N(L_j)$. Note that $S_j$ is not necessarily embedded in $X$ since it may intersect other components of $L$. 

Let $\overline{\Sigma}$ be a compact surface embedded in $S^3$. Then $\overline{\Sigma}$ is a spanning surface for a link $L$ if $\partial\overline{\Sigma}=L$. Let $\Sigma$ be a surface-with-boundary which is properly embedded in $X$. Then $\Sigma$ is a spanning surface for $X$ if each component of $\partial\Sigma$ has minimal geometric intersection number one with the meridian $\mu_j$ of $C_j$. These two notions of spanning surface are related since $\Sigma=X\cap\overline{\Sigma}$ whenever $\overline{\Sigma}$ is in general position with respect to $\partial X$, and $\Sigma$ can be extended to $\overline{\Sigma}$ by attaching a small annulus in each component of $N(L)$.

A pair of spanning surfaces is in general position in $X$ if they intersect in a set of properly embedded arcs and embedded loops. In particular, there are no triple points or branch points, since each spanning surface is properly embedded.

For $j=1,\ldots,m$, let $\{\sigma_j\}$ be the components of $\partial\Sigma$ and let $\{\sigma'_j\}$ be the components of $\partial\Sigma'$.  Fix an orientation on each longitude $\lambda_j$ and define the orientation of each meridian $\mu_j$ so that $(\lambda_j,\mu_j)$ form a right-handed basis for each torus boundary component $C_j$. Then $[\sigma_j]=p_j[\mu_j]+[\lambda_j]$ and $[\sigma'_j]=p'_j[\mu_j]+[\lambda_j]$ where $p_j, p'_j\in\mathbb{Z}$. We define the algebraic intersection number of $\sigma_j$ and $\sigma'_j$ to be 
\[i_a(\sigma_j,\sigma'_j)=p_j-p'_j= -i_a(\sigma'_j,\sigma_j),\]
while the geometric intersection number is 
\[i(\sigma_j,\sigma'_j)=\lvert i_a(\sigma_j,\sigma'_j)\rvert=\lvert p_j-p'_j\rvert.\] 
Define the geometric intersection number of $\partial\Sigma$ and $\partial\Sigma'$ to be
\[i(\partial\Sigma,\partial\Sigma')=\sum\limits_{j=1}^{m}i(\sigma_j,\sigma'_j).\]

This geometric intersection number is measuring the difference in aggregate slopes of the two spanning surfaces, as defined in~\cite{adkind}. For a knot this is the difference in boundary slopes.

If we isotope $\Sigma$ and $\Sigma'$ so that they realise $i(\partial\Sigma,\partial\Sigma')$, then $\sigma_j$ and $\sigma'_j$ form a quadrangulation $Q_j$ of $C_j$, where each quadrangular face has one pair of non-adjacent vertices on its boundary which are identified. There are $i_j=i(\sigma_j,\sigma'_j)$ vertices, $2i_j$ edges, and $i_j$ faces in $Q_j$. We refer to $\bigsqcup_{j=1}^m Q_j$ as a boundary quadrangulation of $\partial X$. When forming a quadrangulation on $C_j$, every intersection of $\sigma_j$ and $\sigma'_j$ has the same sign.

\begin{figure}[htb!]
    \centering
    \includegraphics[scale=0.42, trim = 5mm 5mm 0mm 3mm, clip]{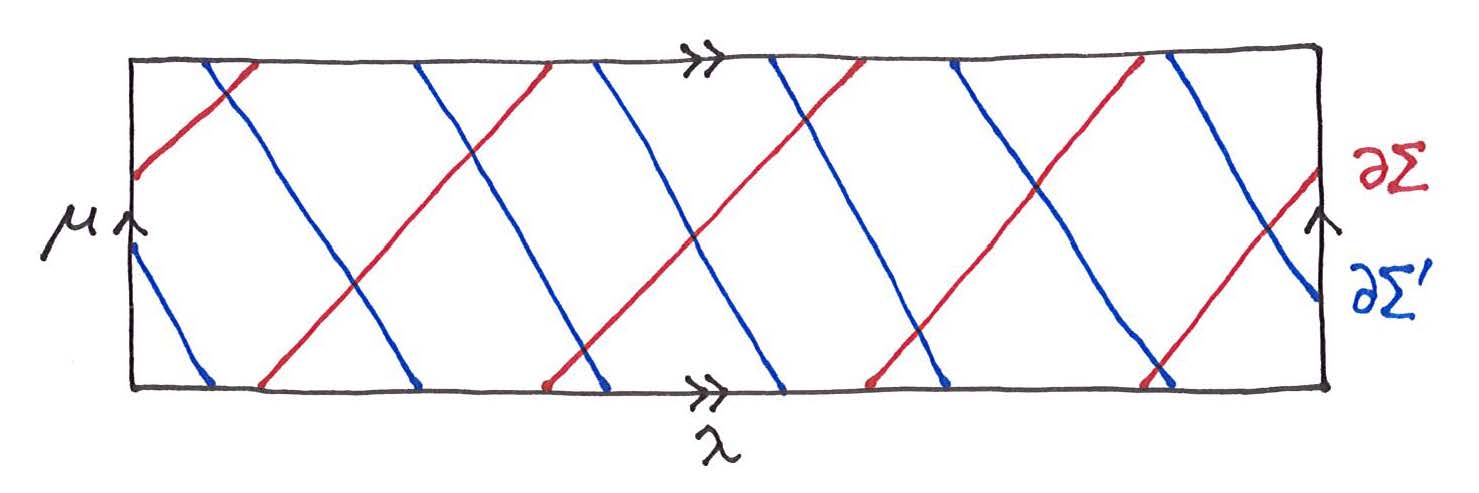}
    \caption{A boundary quadrangulation of $\partial X$ formed by $\partial\Sigma$ and $\partial\Sigma'$ in the case of a knot exterior.}
    \label{bdryquadr}
\end{figure}

An arc of intersection between two spanning surfaces $\Sigma$ and $\Sigma'$ in a link exterior $X$ is called a double arc. Let $\alpha$ be a double arc and let $\overline{\alpha}$ be its extension to $S^3$ such that $\partial\overline{\alpha}\subset L$. There are two types of double arc.

Let $W$ be a regular neighbourhood of $\overline{\alpha}$ in $S^3$. Let $\beta$ and $\beta'$ be the two components of $W\cap L$. We can choose $W$ so that $V=W\cap X$ is a compact handlebody of genus two, and so that $\Sigma\cap V$ and $\Sigma'\cap V$ are both disks. 

Fix an orientation on $\beta$. This induces orientations on the disks $\Sigma\cap V$ and $\Sigma'\cap V$. If these both induce the same orientation on $\beta'$, then $\alpha$ is a parallel arc. If they induce opposite orientations on $\beta'$, then $\alpha$ is a standard arc.

\begin{figure}[htb!]
    \centering
    \includegraphics[scale=0.6, trim = 2mm 5mm 0mm 3mm, clip]{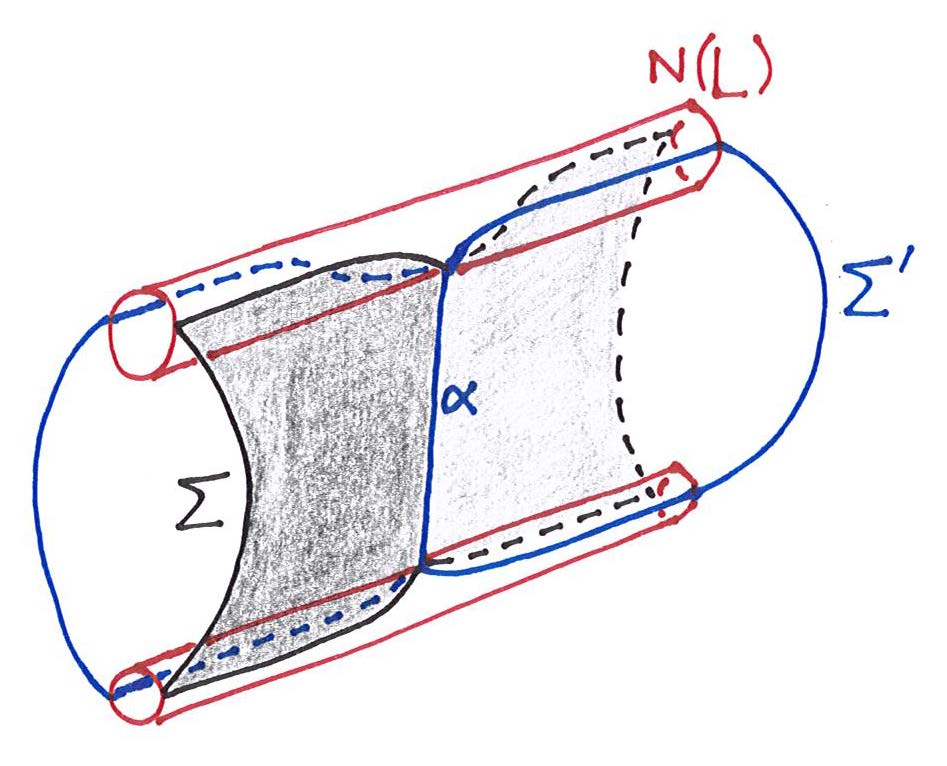}
    \caption[A parallel arc of intersection]{A parallel arc of intersection between two spanning surfaces.}
    \label{paralarch}
\end{figure}

This is equivalent to saying $\partial(\Sigma\cap V)$ and $\partial(\Sigma'\cap V)$ have algebraic intersection number zero if $\alpha$ is a parallel arc, and algebraic intersection number two if $\alpha$ is a standard arc. The intersections on $\partial X$ at the endpoints of a parallel arc $\alpha$ have the same sign if $\alpha$ is standard, but opposite signs when $\alpha$ is parallel.

If we collapse a standard arc to a point, then $(\overline{\Sigma}\cup\overline{\Sigma}')\cap W$ collapses to a disk. If we try to collapse a parallel arc to a point, then $(\overline{\Sigma}\cup\overline{\Sigma}')\cap W$ collapses to an object homeomorphic to a neighbourhood of the apex of a double cone.

\begin{lemma}\label{standarc}
Let $\Sigma$ and $\Sigma'$ be spanning surfaces for a link $L$, isotoped so that their boundaries realise the intersection number $i(\partial\Sigma,\partial\Sigma')$ on $\partial X$. If
\[i(\partial\Sigma,\partial\Sigma')=\lvert\sum\limits_{j=1}^{m}i_a(\sigma_j,\sigma'_j)\rvert,\]
then every double arc of intersection between $\Sigma$ and $\Sigma'$ is standard.
\end{lemma}

\begin{proof}
By definition $i(\partial\Sigma,\partial\Sigma')=\sum\limits_{j=1}^{m} i(\sigma_j,\sigma'_j)$. If
\[\lvert\sum\limits_{j=1}^{m}i_a(\sigma_j,\sigma'_j)\rvert=\sum\limits_{j=1}^{m} i(\sigma_j,\sigma'_j),\]
then for every $j=1,\ldots,m$, either every intersection between $\sigma_j$ and $\sigma'_j$ is positive, or every intersection between $\sigma_j$ and $\sigma'_j$ is negative. A parallel arc only occurs when a double arc connects a positive intersection to a negative intersection.
\end{proof}

Note that Lemma~\ref{standarc} implies that parallel arcs of intersections can only occur when the double arc runs between different components of $\partial X$. Hence if $K$ is a knot, then every arc of intersection between two spanning surfaces realising minimal intersection number must be standard. 

\begin{figure}[htb!]
    \centering
    \includegraphics[scale=1.0, trim = 0mm 7mm 0mm 5mm, clip]{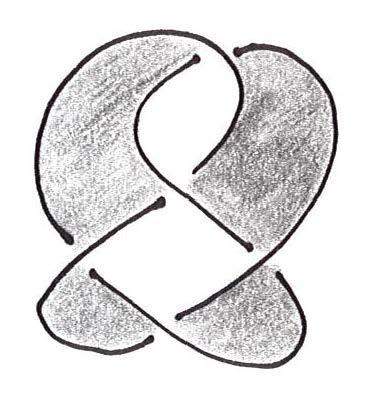}
    \caption{Black and white checkerboard surfaces for a knot}
    \label{cheekysmurf}
\end{figure}

For any non-split planar link diagram $\pi(L)$, there is a standard position for the associated checkerboard surfaces. Away from a crossing the checkerboard surfaces are embedded in $S^2$, but in a small regular neighbourhood of a crossing, we think of the link lying on the surface of a ball $U$. The ball $U$ intersects $S^2$ in an equatorial disk, and the over strand runs over the upper hemisphere, while the under strand runs under the lower hemisphere. Each checkerboard surface intersects $U$ in a half-twisted band. The ball $U$ is called a bubble and this viewpoint of checkerboard surfaces for planar alternating diagrams was introduced by Menasco~\cite{menasc}. 

\begin{figure}[htb!]
    \centering
    \includegraphics[scale=0.5, trim = 5mm 5mm 0mm 3mm, clip]{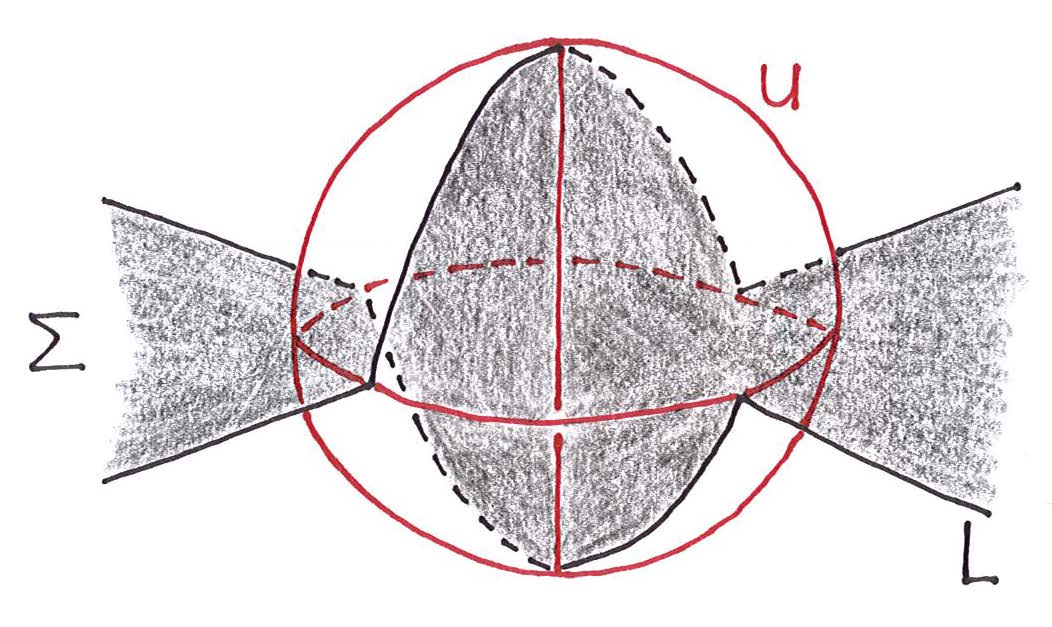}
    \caption{A checkerboard surface in standard position near a bubble $U$.}
    \label{bubtech}
\end{figure}

In standard position, the checkerboard surfaces $\Sigma$ and $\Sigma'$ do not intersect in any loops, and intersect only in double arcs corresponding to the north-south axis of each bubble. These double arcs are all standard. The intersection of the corresponding spanning surfaces $\overline{\Sigma}$ and $\overline{\Sigma}'$ is a disjoint union of trivalent graphs and loops consisting of the link $L$ and the collection of vertical axes of the bubbles. If we assume that $\pi(L)$ is non-split, then every component of $L$ is involved in a crossing of $\pi(L)$, so that $\overline{\Sigma}\cap\overline{\Sigma}'$ forms a graph $\Gamma'$, where each connected component is $3$-regular.

For a non-split alternating planar projection $\pi(L)$ in standard position, it can be seen that as we traverse the image of $L_j$ in $\pi(L)$, $\sigma_j$ rotates in a positive manner say, with respect to $S^2$, while $\sigma'_j$ rotates in a negative manner. Hence the checkerboard surfaces in standard position already realise the minimal geometric intersection number of their boundaries, and their boundaries form a boundary quadrangulation of $\partial X$.

Thus if $\pi(L)$ is a non-split planar alternating diagram of an $m$-component link $L$, which has $n$ crossings, then 
\[2n = i(\partial\Sigma,\partial\Sigma')=\sum\limits_{j=1}^{m}i(\sigma_j,\sigma'_j).\]
A method for calculating the boundary slopes of the checkerboard surfaces associated to a reduced alternating knot diagram is detailed in~\cite{curttay}, where it can be seen that $2n$ is the difference between the boundary slopes of the two checkerboard surfaces. If $\Sigma$ is a spanning surface for a knot and $[\partial\Sigma]=p[\mu]+[\lambda]$, then the boundary slope of $\Sigma$ is $p\in 2\mathbb{Z}$.

Note that if $\pi(L)$ is a non-alternating planar diagram with n crossings, then somewhere there are two consecutive over-crossings, which forces the boundaries of the associated checkerboard surfaces to create a bigon on some component of $C_j$. In this case $n > \frac{1}{2}i(\partial\Sigma,\partial\Sigma')$. 

The Euler characteristics of the checkerboard surfaces arising from a planar projection are related to the Euler characteristic of the projection sphere by the equation
\[\chi(\Sigma)+\chi(\Sigma')+n=\chi(S^2)=2.\] 

A spanning surface $\Sigma$ is $\pi_1$-essential in a knot exterior $X$ if the induced homomorphism 
\[\pi_1(\Sigma)\rightarrow\pi_1(X),\]
and the induced map
\[\pi_1(\Sigma,\partial\Sigma)\rightarrow\pi_1(X,\partial X),\]
are both injective. This implies that $\Sigma$ is both incompressible and boundary-incompressible in $X$.

Aumann~\cite{aumasph} proved that both checkerboard surfaces associated to a planar alternating projection of a knot $K$ are $\pi_1$-essential in $X$. There have been several subsequent proofs and generalisations of this result, such as a theorem of Ozawa~\cite{ozstate}, which is more general than we need, but does include the case of the checkerboard surfaces associated to a planar alternating link diagram which is not necessarily prime.

\subsection{Characterisation}\label{charac}

We now give the proof of the non-diagrammatic characterisation of alternating knot exteriors.

\begin{theorem}\label{chars3}
Let $K$ be a non-trivial knot in $S^3$ with exterior $X$. Then $K$ has an alternating projection onto $S^2$ if and only if there exist a pair of connected spanning surfaces $\Sigma$, $\Sigma'$ for $X$ which satisfy
\[\chi(\Sigma)+\chi(\Sigma')+\frac{1}{2}i(\partial\Sigma,\partial\Sigma')=2.\label{eulereqn}\tag{$\star$}\]
\end{theorem}

\begin{proof}
One direction follows from the discussion in Section~\ref{intsurf}. For the converse, let $\Sigma_0$ and $\Sigma_0'$ be a pair of connected spanning surfaces for $X$ which satisfy~\eqref{eulereqn}.

Since $X$ is not a solid torus, $X$ is boundary-irreducible, so $K$ does not bound a disk in $S^3$. Hence $\chi(\Sigma_0)+\chi(\Sigma'_0)\leq 0$, so by~\eqref{eulereqn}, $i(\partial\Sigma_0,\partial\Sigma'_0)\not=0$.

Isotope $\Sigma_0$ and $\Sigma'_0$ in $X$ so that their boundaries realise the minimal geometric intersection number $i(\partial\Sigma_0,\partial\Sigma'_0)$. Hence $\partial\Sigma_0$ and $\partial\Sigma'_0$ form a boundary quadrangulation $Q$, and $Q$ will remain fixed throughout the proof. We may assume that $\Sigma_0$ and $\Sigma'_0$ are in general position, so that they intersect in a set of proper arcs $\mathcal{A}$ and a set of embedded loops $\mathcal{L}_0$. 

Recall that $\overline{\Sigma}_0$ is the extension of $\Sigma_0$ to $S^3$ so that $\partial\overline{\Sigma}_0=L$. Assume that the interiors of $\overline{\Sigma}_0$ and $\overline{\Sigma}'_0$ are in general position, and no loops of intersection have been introduced by the extension process. Let $F'_0=\overline{\Sigma}_0\cup\overline{\Sigma}'_0$ and let $\Gamma'=(\overline{\Sigma}_0\cap\overline{\Sigma}'_0)\setminus\mathcal{L}_0$, both of which are connected. Let $\overline{\mathcal{A}}$ be the extension of $\mathcal{A}$ to $S^3$, or in other words, let $\overline{\mathcal{A}}$ be the closure of $\Gamma'\setminus L$.

If we collapse each component of $\overline{\mathcal{A}}$ to a point, then $F'_0$ collapses to an immersed surface $F_0$ because $X$ is a knot exterior and every arc of $\mathcal{A}$ is standard by Lemma~\ref{standarc}. Cutting $F'_0$, $\overline{\Sigma}_0$, and $\overline{\Sigma}'_0$ along $\Gamma'$ allows us to calculate that
\[\chi(F'_0) = \chi(\overline{\Sigma}_0)+\chi(\overline{\Sigma}'_0)+\frac{1}{2}i(\partial\Sigma_0,\partial\Sigma'_0),\]
from which the equation~\eqref{eulereqn} tells us that $\chi(F'_0)=2$. The $3$-regular graph $\Gamma'$ collapses to a $4$-regular graph $\Gamma$. Let \[f_0:S_0\looparrowright S^3\] be the immersion of a surface $S_0$ such that $f_0(S_0)=F_0$. There are no triple points of self-intersection in $F_0$ since $\Sigma_0$ and $\Sigma'_0$ are embedded, and the only double loops of self-intersection are precisely the set $\mathcal{L}_0$. It follows that 
\[\chi(S_0)=\chi(F_0)=\chi(F'_0)=2,\]
which implies that $S_0$ is a $2$-sphere.

Suppose $\mathcal{L}_0\not=\emptyset$ and let $\mathcal{B}_0$ be the collection of loops $f_0^{-1}(\mathcal{L}_0)$ on $S_0$. Since $\mathcal{B}_0$ is the pre-image of double loops, we know that $\mathcal{B}_0$ conatins an even number of elements. Because $S_0$ is a $2$-sphere, each loop $\beta\in\mathcal{B}_0$ is separating, and $\mathcal{B}_0$ cuts $S_0$ into a collection of planar surfaces with boundary. Let $z_h$ be the number of planar surfaces in $S_0\setminus\mathcal{B}_0$ which have $h$ boundary components. Exactly one component of $S_0\setminus\mathcal{B}_0$ contains the connected graph $f_0^{-1}(\Gamma)$.

Using an Euler characteristic argument, Nowik~\cite{nowik} points out that 
\[\sum_{h\geq 1}(2-h)z_h=2,\]
which in particular implies that $z_1$, the number of disk regions in $S_0\setminus\mathcal{B}_0$, is at least $2$. Thus there is at least one loop in $\mathcal{L}_0$ which bounds a disk in $F_0$.

Let $\ell_0$ be a loop in $\mathcal{L}_0$ which bounds a disk $D_0$ in either $\Sigma_0$ or $\Sigma'_0$. Without loss of generality assume $D_0\subset\Sigma_0$. Let $\{\beta_0,\beta'_0\}=f_0^{-1}(\ell_0)$ where $f_0(N(\beta_0))\subset\Sigma_0$ and $f_0(N(\beta'_0))\subset\Sigma'_0$. Notice that $f_0^{-1}$ restricted to the interior of $D_0$ is a homeomorphism onto the interior of a disk in $S_0\setminus\mathcal{B}_0$.

Let $A_0=\ell_0\times(-1,1)$ be a regular neighbourhood of $\ell_0$ in $\Sigma'_0$. We perform surgery on $\Sigma'_0$ along $D_0$, by removing the annulus $A_0$ from $\Sigma'_0$ and gluing in the two disks $D_0\times\{-1\}$ and $D_0\times\{1\}$.

\begin{figure}[htb!]
    \centering
    \includegraphics[scale=0.47, trim = 0mm 5mm 0mm 3mm, clip]{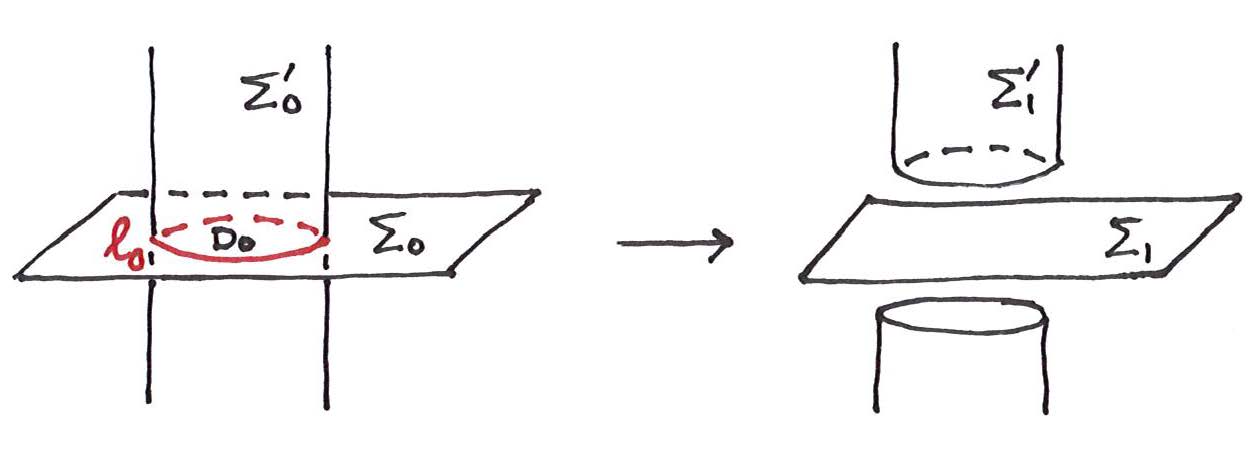}
    \caption[Surgery along a disk]{Surgery on $\Sigma'_0$ along the disk $D_0\subset\Sigma_0$.}
    \label{disksurg}
\end{figure}

Let $(\Sigma_1,\Sigma'_1)$ be the result of performing surgery along $D_0$ on $(\Sigma_0,\Sigma'_0)$. Define $S_1$ to be the result of surgery along $f_0^{-1}(D_0)$ in $S_0$ which removes the loop $\beta'_0$, and deleting the curve $\beta_0$. 

Suppose $\lvert\mathcal{L}_0\rvert=k\geq0$. For $0\leq j\leq k-1$, inductively define $(\Sigma_{j+1},\Sigma'_{j+1})$ to be the result of performing surgery along the disk $D_j$, where $\partial D_j=\ell_j\subset\mathcal{L}_j$ and $D_j$ is a sub disk of either $\Sigma_j$ or $\Sigma'_j$. Let $\mathcal{L}_{j+1}=\mathcal{L}_j\setminus\ell_j$ so that $\mathcal{L}_{j+1}$ is the set of loops of intersection of $\Sigma_{j+1}$ and $\Sigma'_{j+1}$. Let $F'_{j+1}$ be the result of the corresponding surgery on the $2$-complex $F'_j$. Define $S_{j+1}$ to be surgery along $f_j^{-1}(D_j)$ in $S_j$ and deleting the curve $\beta'_j$ or $\beta_j$. 

A similar calculation to Nowik's shows that
\[\sum_{h\geq 0}(2-h)z_h=\chi(S_j)=2+2j,\]
where $S_j$ is a collection of $(j+1)$ closed $2$-spheres since $\beta_j$ and $\beta'_j$ are separating in $S_j$. Since $\Gamma$ is connected and disjoint from $\mathcal{L}_0$, $f_j^{-1}(\Gamma)$ is contained in exactly one component of $S_j$. Hence $z_0\leq j$, so that $z_1\geq 2$, and therefore the disk $D_j$ exists.

Continue this inductive surgery process until we have constructed $\Sigma_k$ and $\Sigma'_k$. At this stage $\mathcal{L}_k$ is empty, and $S_k$ consists of $(k+1)$ $2$-spheres. Hence $F'_k$ consists of $k$ unmarked embedded 2-spheres and one $2$-complex, denoted $F'$, which contains $\Gamma'$.

Define $\overline{\Sigma}$ and $\overline{\Sigma}'$ to be the components of $\overline{\Sigma}_k$ and $\overline{\Sigma}'_k$ respectively which constitute $F'$. Then $\Sigma$ and $\Sigma'$ are connected spanning surfaces for $X$ which satisfy~\eqref{eulereqn}, and whose intersection is exactly $\mathcal{A}$. Collapsing the arcs of $\overline{\mathcal{A}}$ to points, collapses $F'$ and $\Gamma'$ to $F$ and $\Gamma$ respectively. Since $\chi(F')=2$, it follows that $\chi(F)=2$ so $F$ is an embedded $2$-sphere, which will be our desired projection surface.

Since $\partial\Sigma$ and $\partial\Sigma'$ realise $i(\partial\Sigma,\partial\Sigma')=i(\partial\Sigma_0,\partial\Sigma'_0)$, we can recover the crossing information of $\pi(K)$ from $\Gamma\subset F$. This is because every double arc is standard, so instead of collapsing every arc of $\overline{\mathcal{A}}$ to a point, we could collapse a regular neighbourhood of each $\overline{\alpha}$ to a bubble. The diagram $\pi(K)$ must be alternating since otherwise, there would be a bigon between $\partial\Sigma$ and $\partial\Sigma'$ on $\partial X$, which contradicts that the boundary quadrangulation has remained fixed. Note that $\pi(K)$ is not necessarily reduced.
\end{proof}

If $\pi(K)$ is a non-alternating planar diagram with checkerboard surfaces $\Sigma$ and $\Sigma'$, then there is a bigon on $\partial X$ between $\partial\Sigma$ and $\partial\Sigma'$. Any attempt to isotope the checkerboard surfaces to remove all bigons and obtain a boundary quadrangulation causes $\Sigma$ and $\Sigma'$ to create an alternating diagram onto a non-planar surface $F$. This surface $F$ could be either embedded or immersed, and in the latter case may even be non-orientable~\cite{joshphd}. 

Note that Theorem~\ref{chars3} is not concerned with primeness. However a knot is prime if there are no essential annuli properly embedded in $X$ at meridional slope. If a non-trivial knot is prime and alternating, then $X$ is atoroidal, since no prime alternating knot is satellite~\cite{menasc}.

In~\cite{joshphd} using a different method, the author also proved a variation of Theorem~\ref{chars3} which required both spanning surfaces to be $\pi_1$-essential in $X$. We also note that a more complicated characterisation can be obtained as a special case of a theorem proved in~\cite{joshphd} which gives a topological characterisation of a class of links which have certain alternating diagrams onto orientable surfaces of higher genus. This will be written up in a forthcoming article with Rubinstein~\cite{joshhyam}.

We have not stated any of the theorems in this section for links. The issue is that given two spanning surfaces, there could be parallel arcs between different components of $\partial X$. If the two spanning surfaces have been isotoped to create a boundary quadrangulation and there are parallel arcs of intersection, then the complex $F'_0$ does not collapse to a surface. We note that there exists an example~\cite{joshphd} of a pair of spanning surfaces for a non-split $2$-component link which intersect in parallel arcs, yet still satisfy equation~\eqref{eulereqn}.

However, if we assume that all arcs of intersection are standard, then we have the following theorem.

\begin{theorem}\label{chars4}
Let $L$ be a non-trivial non-split link in $S^3$ with exterior $X$ which has a marked meridian on each boundary component. Then $L$ has an alternating projection onto $S^2$ if and only if there exist a pair of connected spanning surfaces $\Sigma$, $\Sigma'$ for $X$ which satisfy
\[\chi(\Sigma)+\chi(\Sigma')+\frac{1}{2}i(\partial\Sigma,\partial\Sigma')=2,\tag{$\star$}\]
and 
\[i(\partial\Sigma,\partial\Sigma')=\lvert\sum_{j=1}^{m}i_a(\sigma_j,\sigma'_j)\rvert.\]
\end{theorem}

\begin{proof}
Let $\pi(L)$ be a reduced non-split alternating projection of $L$ onto $S^2$, and let $\Sigma$ and $\Sigma'$ be the associated checkerboard surfaces in standard position. Then $\pi(L)$, $\Sigma$, and $\Sigma'$ are connected, so every arc of intersection between $\Sigma$ and $\Sigma'$ is standard. This means every algebraic intersection number $i_a(\sigma_j,\sigma'_j)$ is either positive or every $i_a(\sigma_j,\sigma'_j)$ is negative, which implies the second equation. The first equation follows from Section~\ref{intsurf}.

For the converse, the concept of a spanning surface for $X$ is well-defined since a set of meridians for $X$ are specified. Lemma~\ref{standarc} ensures that every arc of intersection is standard. Then the rest of the proof goes through as in Theorem~\ref{chars3}. Since $\Sigma$, and $\Sigma'$ are connected, $\pi(L)$ must be non-split.
\end{proof}

\subsection{Normal Surface Theory}\label{normsurf}

In this section we provide the background material necessary to construct our algorithm. Kneser introduced the concept of a normal surface, before Haken~\cite{haknf} developed normal surface theory into an important tool for algorithmic topology. We will give a brief outline of the theory, for full details the reader is referred to~\cite{jacotol}.

A knot manifold is a compact irreducible $3$-manifold with connected torus boundary. A triangulation $\mathcal{T}$ of a knot manifold $M$ is a collection of $t$ tetrahedra and a set of equations which identify some pairs of faces of the tetrahedra, so that the link of every vertex is either a $2$-sphere or a disk, and the unglued faces form the boundary torus $\partial M$. 

A normal surface $S$ is a properly embedded surface in $M$ which is tranverse to the 2-skeleton of $\mathcal{T}$, and such that $S\cap\triangle$ is a collection of triangular or quadrilateral disks, where $\triangle$ is any tetrahedron of $\mathcal{T}$, and each disk intersects each edge of $\triangle$ in at most one point. There are seven normal isotopy classes of normal disks, four are triangular and three are quadrilateral, and each is these is known as a disk type. 

If we fix an ordering of the disk types $d_1, d_2,..., d_{7t}$, then a normal surface $S$ can be represented uniquely up to normal isotopy by a $7t$-tuple of non-negative integers $\mathbf{n}(S)= (x_1, x_2,..., x_{7t})$, where $x_i$ is the number of disks of type $d_i$, and $t$ is the number of tetrahedra in $\mathcal{T}$.  

Conversely, given a $7t$-tuple of non-negative integers $\mathbf{n}$, we can impose restrictions on the $x_i$ so that $\mathbf{n}$ represents a properly embedded normal surface. We require that at least two of the three quadrilateral disk types are not present in each tetrahedra. This ensures that the surface is embedded. We also need to make sure that the disk types match up with the disk types in neighbouring tetrahedra.

An arc type is the normal isotopy class of the intersection of a normal surface with a face of a tetrahedron. There are three arc types in each face of each tetrahedron, and each arc type is contributed to by two different disk types, one triangular, the other quadrilateral. We require that the number of each arc type in each face agrees with the number of arcs of the corresponding type in the face of the tetrahedron which is glued to it. This condition can be described by a linear equation for each arc type. Together they are called the matching equations for the normal surface $S$, and in a one-vertex triangulation of a knot manifold $M$, there are $6t-3$ matching equations.

The set of non-negative integer solutions to the normal surface equations lie within an infinite linear cone $\mathcal{S}_{\mathcal{T}}\subset\mathbb{R}^{7t}$. The linear cone $\mathcal{S}_{\mathcal{T}}$ is called the solution space.

The additional condition that 
\[\sum_{i=1}^{7t} x_i = 1,\] 
turns the solution space into a compact, convex, linear cell $\mathcal{P}_{\mathcal{T}}\subset\mathcal{S}_{\mathcal{T}}$. We call $\mathcal{P}_{\mathcal{T}}$  the projective solution space, and we let $\hat{\mathbf{n}}(S)$ represent the projective class of the normal surface $S$. The carrier of a normal surface $S$, denoted $\mathcal{C}_{\mathcal{T}}(S)$, is defined to be the unique minimal face of $\mathcal{P}_{\mathcal{T}}$ which contains $\hat{\mathbf{n}}(S)$.

Let $S$ be a properly embedded surface in a 3-manifold $M$ with triangulation $\mathcal{T}$. Haken~\cite{haknf} showed that after a series of isotopies, compressions, boundary-compressions and the removal of trivial $2$-spheres and disks, $S'$ can be represented as the union of properly embedded normal surfaces with respect to $\mathcal{T}$. In particular, if $S$ is $\pi_1$-essential in $M$, then $S$ can be isotoped to be normal with respect to $\mathcal{T}$.

Two normal surfaces $S$ and $S'$ are compatible if for each tetrahedron $\triangle$ of $\mathcal{T}$, $S$ and $S'$ do not contain quadrilateral disks of different types. If $S$ and $S'$ are compatible, then we can form the Haken sum of $S$ and $S'$, which we denote $S\oplus S'$. The Haken sum is a geometric sum along each arc and loop of intersection between $S$ and $S'$, which is uniquely defined by the requirement that $S\oplus S'$ is a normal surface. Any other choice of geometric sums would produce a surface with folds. If $\mathbf{n}(S)=(x_1, x_2,\ldots, x_{7t})$ and $\mathbf{n}(S')=(x'_1, x'_2,\ldots, x'_{7t})$ are representatives of compatible normal surfaces $S$ and $S'$ in a triangulation $\mathcal{T}$ of a $3$-manifold $M$, then $\mathbf{n}(S\oplus S')=\mathbf{n}(S)+\mathbf{n}(S') = (x_1+x'_1, x_2+x'_2,\ldots, x_{7t}+x'_{7t})$. Also, $\chi(S\oplus S')=\chi(S)+\chi(S')$.

A normal surface $S$ is called a vertex surface if $\hat{\mathbf{n}}(S)$ lies at a vertex of the projective solution space. This means that whenever some multiple of $S$ can be written as a Haken sum of two surfaces, then both the summands are also multiples of $S$.

A normal surface $S$ is called a fundamental surface if $\mathbf{n}(S)$ cannot be written as the sum of two solutions to the normal surface equations. Every vertex surface is a fundamental surface, but there exist fundamental surfaces which are not vertex surfaces.

All normal surfaces can be written as a finite sum of fundamental surfaces. There are a finite number of fundamental surfaces. They can be found algorithmically and Haken used this fact to construct his algorithms. Many of these algorithms have been subsequently improved so that they utilise vertex solutions rather than fundamental solutions, which makes the algorithms more efficient.

In particular, Jaco and Tollefson~\cite{jacotol} were able to show that if a knot-manifold contains a properly embedded normal surface $S$ with $\chi(S)=0$, then there exists such a surface at a vertex of $\mathcal{P}_{\mathcal{T}}$.

A triangulation is $0$-efficient if the only normal disks or normal spheres are vertex-linking. Jaco and Rubinstein~\cite{0efftri} showed that every compact orientable irreducible and boundary-irreducible $3$-manifold with non-empty boundary admits a $0$-efficient triangulation. Since the solid torus admits a one-vertex triangulation, it then follows that every knot exterior admits a one-vertex triangulation.

Let $\mathcal{T}$ be a one-vertex triangulation of a knot manifold $M$. Then there is an induced one-vertex triangulation of $\mathcal{T}_{\partial}$ of $\partial M$. The boundary triangulation $\mathcal{T}_{\partial}$ consists of one vertex, three edges, and two faces. There are six normal arc types, however a normal curve can is determined by just three of these arc types. Every curve on $\partial M$ has a unique normal representative. This means that isotopy classes of curves on $\partial M$ correspond to normal isotopy classes of curves on $\mathcal{T}_{\partial}$.

\begin{figure}[htb!]
    \centering
    \includegraphics[scale=0.6, trim = 0mm 12mm 0mm 7mm, clip]{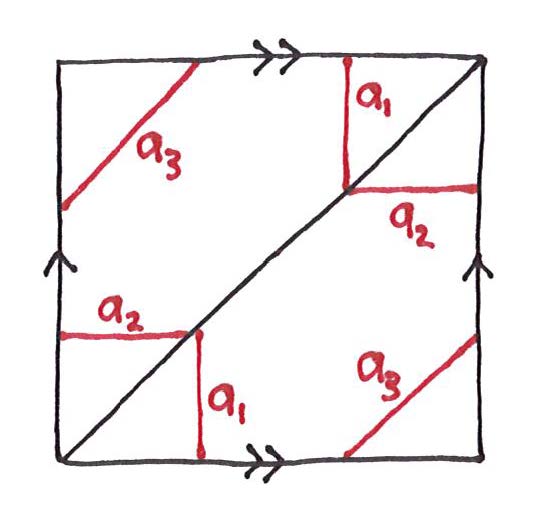}
    \caption[Normal arc types in the boundary triangulation]{Normal arc types in the boundary triangulation.}
    \label{archetype}
\end{figure}

Fix an ordering of the disk types in $\mathcal{T}$ such that $d_1,\ldots,d_7$ represent the disk types in one of the tetrahedra which meets $\partial M$ in a face $\phi$. Furthermore, let $d_1,\ldots,d_4$ represent triangular disk types, and $d_5,d_6,d_7$ represent quadrilateral disks, such that $d_i$ and $d_{i+4}$ meet $\phi$ in the same arc type $a_i$ for $i=1,2,3$. Here $d_4$ is the triangular disk type which is disjoint from $\phi$.

Let $y_i$ be the number of arcs of type $a_i$ in $\phi$. It follows that 
\[y_i=x_i+x_{i+4}\] for $i=1,2,3$. Jaco and Segwick~\cite{jacosed} showed that $y_1,y_2,y_3$ and the matching equations for normal curves determine the number of arcs of each type in the other $2$-simplex of $\mathcal{T}_{\partial}$. We define the boundary solution space of $\partial M$ to be 
\[\mathcal{S}_{\mathcal{T}_{\partial}}=\{(y_1,y_2,y_3)|y_i\in\mathbb{N}_0\}\subset\mathbb{R}^3,\]
where $\mathbb{N}_0=\mathbb{N}\cup\{0\}$. If $\partial S$ is the boundary of a properly embedded normal surface, then $\partial S$ is represented by $\mathbf{n}(\partial S)=(y_1,y_2,y_3)$ in $\mathcal{S}_{\mathcal{T}_{\partial}}$.

If each coordinate of $\mathbf{n}(\partial S)$ is non-zero, then $\partial S$ contains a trivial curve. Hence if $S$ is an incompressible surface, then at least one of the coordinates of $\mathbf{n}(\partial S)$ is zero.

Jaco and Sedgwick~\cite{jacosed} proved that if $S$ is a properly embedded $\pi_1$-essential normal surface with boundary in a compact $3$-manifold $M$ with triangulation $\mathcal{T}$, then every surface in $\mathcal{C}_{\mathcal{T}}(S)$ is either closed, or has the same slope as $S$. This means that if $\frac{p}{q}$ is a boundary slope of $X$, then there is a vertex surface $S$ which has slope $\frac{p}{q}$. Hence it is only necessary to check the vertices of $\mathcal{P}_{\mathcal{T}}$ in order to list all boundary slopes of $X$. In proving this theorem, Jaco and Sedgwick have given another proof of a theorem of Hatcher~\cite{hatbc} that there are only a finite number of slopes bounding $\pi_1$-essential surfaces in any knot exterior. Recall that the set of slopes of a link exterior which bound $\pi_1$-essential surfaces is not necessarily finite, so our algorithm is only designed to work for knots.

Jaco and Sedgwick~\cite{jacosed} also gave an algorithm to decide if a knot manifold $M$ is a knot exterior in $S^3$, which is also an algorithm to find the unique meridian $\mu$ of a knot exterior $X$. This algorithm makes use of the Rubinstein-Thompson algorithm~\cite{rub3sph, thomp3} which can decide if a $3$-manifold is homeomorphic to the $3$-sphere.  There is also an algorithm which can decide if $X$ is a solid torus~\cite{haknf}, which is equivalent to deciding if $K$ is the unknot, and it is now known that some spanning disk for $K$ can be found as a vertex solution if $K$ is the unknot.

The boundary triangulation $\mathcal{T}_{\partial}$ consists of two $2$-simplices and three edges. We can modify the triangulation $\mathcal{T}$ by gluing two faces of a tetrahedron $\triangle$ to $\mathcal{T}_{\partial}$. The resulting triangulation $\mathcal{T}'=\mathcal{T}\cup\triangle$ is another one vertex triangulation of $M$, and $\mathcal{T}'$ is called a layered triangulation. In effect, this is a $(2,2)$-Pachner move on $\mathcal{T}_{\partial}$. The other two faces of $\triangle$ form the boundary triangulation $\mathcal{T}'_{\partial}$.

Layering a tetrahedron changes the slope of one of the edges in the boundary triangulation. It is always possible to layer a triangulation with a sequence of tetrahedra so that the edges of $\mathcal{T}'_{\partial}$ have slopes $\infty, k, k+1$, for some $k\in\mathbb{Z}$. We will choose to do this so that $(1,0,0)\in\mathcal{S}_{\mathcal{T}_{\partial}}$ represents the meridian $\mu$.

\subsection{Alternating Algorithm}\label{altalg}

We now describe an algorithm to decide if a knot is alternating on $S^2$. The input is a triangulation $\mathcal{T}$ of a knot exterior $X$. If instead we are given a non-alternating planar diagram $\pi(K)$, then there is a method of Petronio~\cite{petroid} to construct a spine of the knot complement $S^3\setminus K$ from the diagram $\pi(K)$. Dual to this spine is an ideal triangulation of $S^3\setminus K$. We can then use an inflation of Jaco and Rubinstein~\cite{jrinfla} to construct a one-vertex triangulation of the knot exterior $X$ from the ideal triangulation.

First we need one more result on spanning surfaces.

\begin{lemma}\label{spandis}
Suppose that $\Sigma$ and $\Sigma'$ are spanning surfaces for a knot $K\subset S^3$. If
\[\chi(\Sigma)+\chi(\Sigma')+\frac{1}{2}i(\partial\Sigma,\partial\Sigma')>2,\]
then at least one of $\Sigma$ or $\Sigma'$ is not connected.
\end{lemma}

\begin{proof}
Isotope $\Sigma$ and $\Sigma'$ so that they are in general position and realise $i(\partial\Sigma,\partial\Sigma')$ on $\partial X$. As in the proof for Theorem~\ref{chars3}, we form the pseudo $2$-complex $F'=\overline{\Sigma}\cup\overline{\Sigma}'$ and collapse the arcs of intersection to points to obtain an immersed surface $F$, where $\chi(F)=\chi(\Sigma)+\chi(\Sigma')+\frac{1}{2}i(\partial\Sigma,\partial\Sigma')>2$. Let $f:S\looparrowright S^3$ be an immersion of a closed surface $S$ such that $f(S)=F$. The only possible self-intersections of $f(S)$ are loops, so $\chi(S)=\chi(F)>2$, and therefore $S$ is not connected. Hence either $\Sigma$ or $\Sigma'$ is not connected.
\end{proof}

\begin{theorem}\label{algdec}
Let $X$ be the exterior of a knot $K\subset S^3$. Given $X$, there is an algorithm to decide if $K$ is prime and alternating.
\end{theorem}

\begin{proof}
Let $\mathcal{T}'$ be a one-vertex triangulation of $X$. As shown in~\cite{jacosed}, any other triangulation of $X$ can be modified to a one-vertex triangulation.

Use the Jaco-Sedgwick algorithm to find the unique meridian $\mu$ of $X$. Included in this process is a check whether $X$ is a solid torus. If $X$ is a solid torus, then $K$ is the unknot which by our convention, is not alternating. 

The theorem of Jaco and Tollefson~\cite{jacotol} tells us that if $X$ contains an incompressible torus $T$, then $\hat{\mathbf{n}}(T)$ must be a vertex of $\mathcal{P}_{\mathcal{T}'}$.
Check whether any of the vertices of $\mathcal{P}_{\mathcal{T}'}$ are tori. If any such tori are not boundary parallel, then $K$ is a satellite knot, but a theorem of Menasco~\cite{menasc} tells us that a prime alternating knot cannot be a satellite knot. Thus we can assume that the only incompressible torus is boundary parallel. It can be decided if a torus $T$ is boundary parallel by cutting $X$ along $T$ and testing whether one of the components is homeomorphic to $T\times I$.

Layer the triangulation until one of the edges in the boundary is parallel to $\mu$. Then the other edges in the boundary are parallel to $\lambda+k\mu$ and $\lambda+(k+1)\mu$ for some $k\in\mathbb{Z}$. Call this triangulation $\mathcal{T}$.

Let $\triangle$ be one of the two tetrahedra that meets the boundary and let $\phi$ be a face of $\triangle$ which lies in the boundary. Let $(x_1,x_2,x_3,x_4,x_5,x_6,x_7)$ describe the normal coordinates of $S\cap\triangle$ where $S$ is a properly embedded surface with boundary in $X$. Let $(y_1,y_2,y_3)$ describe the normal coordinates of $\partial S\cap\phi$. As described in Section~\ref{normsurf}, we label the arc and disc types such that $y_i = x_i+x_{i+4}$ for each $i = 1,2,3$, and such that $(1,0,0)$ represents $\mu$ in $\partial X$. Let $(0,1,0)$ represent $\lambda+k\mu$ so that $(0,0,1)$ represents $\lambda+(k+1)\mu$ for some $k\in\mathbb{Z}$.

\begin{figure}[htb!]
    \centering
    \includegraphics[scale=0.45, trim = 0mm 12mm 0mm 7mm, clip]{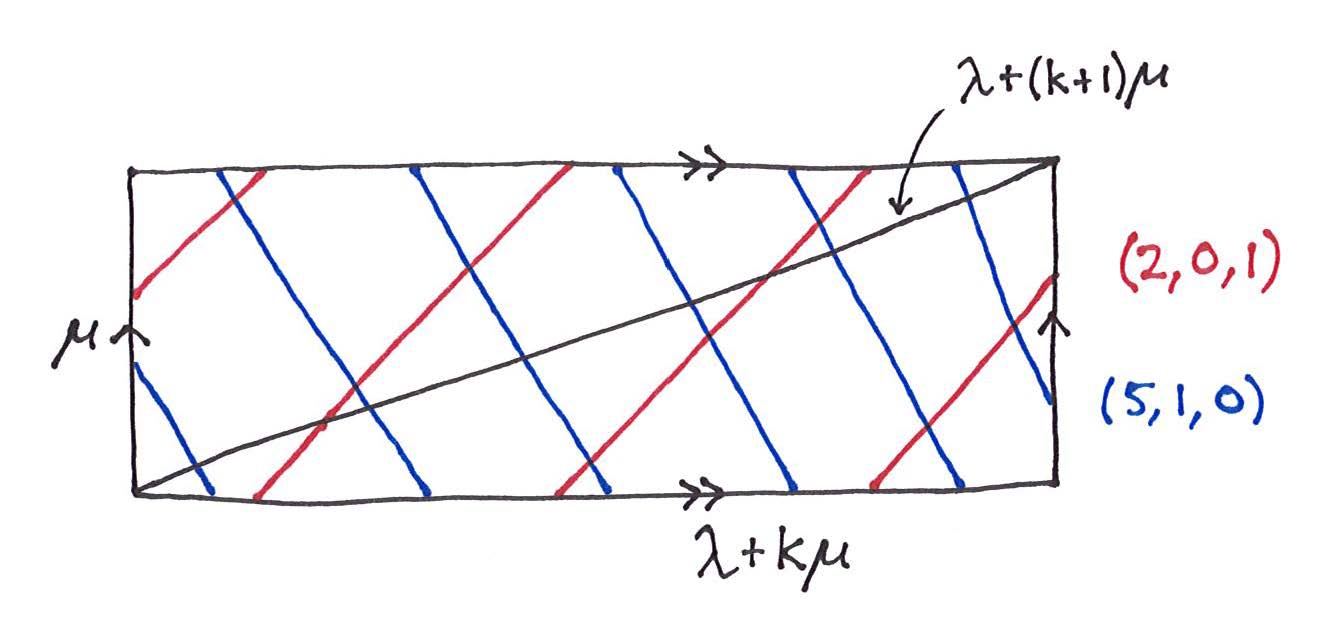}
    \caption[Normal curves in the boundary triangulation after layering]{Normal curves in the boundary triangulation after layering.}
    \label{bdintcnt}
\end{figure}

Any spanning surface for $X$ meets $\mu$ exactly once. It follows that if $(y_1,y_2,y_3)$ represents a spanning surface, then $y_2+y_3=1$. So there are two types of coordinates in $\mathcal{S}_{\mathcal{T}_{\partial}}$ which can represent spanning surfaces: $(y,1,0)$ and $(y,0,1)$ for some $y\in\mathbb{N}_0$.

Let $\Sigma$ and $\Sigma'$ be normal spanning surfaces in $X$. We can read off the minimal geometric intersection number of their boundaries from their coordinates in $\mathcal{S}_{\mathcal{T}_{\partial}}$. Let $y,y'\in\mathbb{N}_0$. There are three cases:
\begin{enumerate}
\item If $\partial\Sigma$ and $\partial\Sigma'$ are represented by $(y,1,0)$ and $(y',1,0)$ respectively, then $i(\partial\Sigma,\partial\Sigma')=|y-y'|$.
\item If $\partial\Sigma$ and $\partial\Sigma'$ are represented by $(y,0,1)$ and $(y',0,1)$ respectively, then $i(\partial\Sigma,\partial\Sigma')=|y-y'|$.
\item If $\partial\Sigma$ and $\partial\Sigma'$ are represented by $(y,1,0)$ and $(y',0,1)$ respectively, then $i(\partial\Sigma,\partial\Sigma')=y+y'+1$. See Figure~\ref{bdintcnt} for an example of this case.
\end{enumerate} 

Note that we could continue layering the triangulation until $k=0$, which would require detection of a Seifert surface, but this is not necessary since we are only interested in the differences of spanning slopes, and not the boundary slopes themselves.

Theorem~\ref{chars3} tells us that we need to find a pair of connected spanning surfaces at even boundary slope, which satisfy
\[\chi(\Sigma)+\chi(\Sigma')+\frac{1}{2}i(\partial\Sigma,\partial\Sigma')=2,\tag{$\star$}\] 
The checkerboard surfaces $\Sigma$ and $\Sigma'$ associated to a reduced alternating diagram of $K$ are one such air of surfaces. Aumann~\cite{aumasph} showed that they are both $\pi_1$-essential in $X$, so we know that both $\Sigma$ and $\Sigma'$ must have normal representatives in their isotopy classes.

Let $\Sigma$ and $\Sigma'$ be a pair of connected normal spanning surfaces which satisfy~\eqref{eulereqn}. Suppose that $\mathbf{n}(\Sigma)$ is not a fundamental solution. Then 
\[\Sigma = \Sigma_1\oplus\ldots\oplus\Sigma_a\oplus S_1\oplus\ldots\oplus S_b,\]
where each $\Sigma_i$ is a properly embedded compact surface with boundary, and each $S_j$ is a properly embedded closed surface. 

The theorem of Jaco and Sedgwick~\cite{jacosed} tells us that each $\Sigma_i$ must have the same slope as $\Sigma$. Since $\Sigma$ is a spanning surface, then $a$ must equal one, and $\Sigma_1$ is also a spanning surface at the same slope as $\Sigma$. In fact, $\Sigma_1$ must be fundamental. 

Since $X$ is irreducible, atoroidal and embedded in $S^3$, it follows that $\chi(S_j)\leq-2$ for each $j=1,\dots,b$. Here we note that if $\Sigma$ is a normal spanning surface for $X$, and $T$ is a boundary-parallel torus, then $\Sigma\oplus T$ is isotopic in $X$ to $\Sigma$, so Haken sum with $T$ can be ignored. 

But then $\chi(\Sigma)<\chi(\Sigma_1)$ which implies that
\[\chi(\Sigma_1)+\chi(\Sigma')+\frac{1}{2}i(\partial\Sigma_1,\partial\Sigma')>2.\]
Hence, by Lemma~\ref{spandis}, at least one of $\Sigma_1$ or $\Sigma'$ must be disconnected. Every fundamental surface is connected, so $\Sigma'$ must be disconnected. 

In that case, 
\[\Sigma' = \Sigma'_1\oplus S'_1\oplus\ldots\oplus S'_c,\]
where $\Sigma'_1$ is fundamental and each $S'_i$ is a closed embedded normal surface. Repeating the previous argument shows that one of $\Sigma_1$ or $\Sigma'_1$ is not connected, contradicting that they are both fundamental. Therefore $\Sigma$ and $\Sigma'$ are fundamental surfaces. 

Let $\mathcal{F}$ be the set of all fundamental spanning surfaces in $X$. For each pair of surfaces $\Sigma,\Sigma'\in\mathcal{F}$, calculate the intersection number $i(\partial\Sigma,\partial\Sigma')$, and calculate $\chi(\Sigma)$ and $\chi(\Sigma')$. There is an algorithm to compute the Euler characteristic of a properly embedded normal surface described in~\cite{jacotol}. If $\Sigma$ and $\Sigma'$ satisfy equation~\eqref{eulereqn}, then $K$ is alternating by Theorem~\ref{chars3}. Since we know $K$ is not a satellite knot, $K$ is prime. If no pair of surfaces from $\mathcal{F}$ satisfy equation~\eqref{eulereqn}, then $K$ is not alternating.
\end{proof}

Let $\pi(K)$ be an alternating diagram of the prime knot $K$ with associated checkerboard surfaces $\Sigma$ and $\Sigma'$. Let $\pi_*(K)$ be a different alternating diagram of $K$ with associated checkerboard surfaces $\Sigma_*$ and $\Sigma'_*$. If $\pi(K)$ and $\pi_*(K)$ are both reduced, then we know from a theorem of Menasco and Thistlethwaite~\cite{menthis}, that $\pi(K)$ and $\pi_*(K)$ are related by a sequence of flypes. In that case, $\Sigma$ and $\Sigma_*$ are homeomorphic and have the same boundary slope, but $\Sigma$ and $\Sigma_*$ may not be isotopic in $X$. The same is true for $\Sigma'$ and $\Sigma'_*$.

However, every checkerboard surface for a reduced alternating diagram is $\pi_1$-essential, and thus will appear amongst our collection of fundamental spanning surfaces $\mathcal{F}$. The collection $\mathcal{F}$ may also contain some pairs of surfaces which correspond to an alternating diagram which is not reduced. In this case, at least one of the checkerboard surfaces fails to be $\pi_1$-essential.

Let $\Sigma$ and $\Sigma'$ have minimal intersection number amongst all surfaces from $\mathcal{F}$ which satisfy~\eqref{eulereqn}. Place an orientation on $\partial\Sigma$, and label the vertices of $\partial\Sigma\cap\partial\Sigma'$ in the order they are encountered as one traverses $\partial\Sigma$ by $1,\ldots,i$ where $i=i(\partial\Sigma,\partial\Sigma')$. Then each arc of intersection between $\Sigma$ and $\Sigma'$ is labelled by two numbers, one even and one odd. These pairs of numbers, listed as a sequence of even positive integers in the order of their paired odd numbers, correspond to the Dowker-Thistlethwaite notation~\cite{dowthis} of a planar alternating diagram of $K$ or its mirror image. 

Therefore, given a non-alternating planar diagram of a knot $K$, there is an algorithm to decide if $K$ is prime and alternating, and if so, there is an algorithm to produce an alternating diagram of $K$ up to chirality.

\bibliography{charalt}
\bibliographystyle{plain}

\noindent \url{j.howie@student.unimelb.edu.au} 

\noindent School of Mathematics and Statistics

\noindent University of Melbourne 

\noindent VIC 3010

\noindent Australia

\end{document}